\documentclass[12pt,reqno]{amsart}
\usepackage{amsmath}
\usepackage{amsfonts}
\usepackage{amssymb}
\usepackage{amscd}
\usepackage{hyperref}

\renewcommand{\bot}{\perp}
\newtheorem{theorem}{Theorem}[section]
\newtheorem{corollary}[theorem]{Corollary}
\newtheorem{lemma}[theorem]{Lemma}
\newtheorem{proposition}[theorem]{Proposition}

\newtheorem{remark}[theorem]{Remark}

\newtheorem{example}[theorem]{Example}

\numberwithin{equation}{section}
\begin{document}
\title{Singular Riemannian flows and characteristic numbers}
\author{Igor Prokhorenkov}
\author{Ken Richardson}
\address{Department of Mathematics\\
Texas Christian University\\
Box 298900 \\
Fort Worth, Texas 76129}
\email{i.prokhorenkov@tcu.edu\\
k.richardson@tcu.edu}
\date{December, 2017}

\begin{abstract}
Let $M$ be an even-dimensional, oriented closed manifold. We show that the
restriction of a singular Riemannian flow on $M$ to a small tubular neighborhood
of each connected component of its singular stratum is
foliated-diffeomorphic to an isometric flow on the same neighborhood. We
then prove a formula that computes characteristic numbers of $M$ as the sum
of residues associated to the infinitesimal foliation at the components of the singular
stratum of the flow.
\end{abstract}

\maketitle

\section{Introduction}

In \cite{BaumCh}, P. Baum and J. Cheeger give a formula for the
characteristic numbers of a manifold endowed with a singular isometric flow
in terms of integrals of characteristic forms over the singular stratum of
the flow. In our paper, we start with a singular Riemannian flow, not
assumed to be an isometric flow. Our first result describes the structure of
singular Riemannian flows near the singular stratum of the flow. Using this,
we obtain a formula similar to \cite{BaumCh} for the characteristic numbers
of the manifold.

The following definitions and facts about Riemannian foliations can be found
in \cite{Mo}. Let $\left( M,\mathcal{F}\right) $ be a singular foliation of
a smooth, connected, compact manifold $M$ that is smooth in the sense of
Sussman and Stefan (\cite{Stef}, \cite{Suss}). This means that for each leaf 
$L\in \mathcal{F}$, each $x\in L$, and each $v\in T_{x}L$, there exists a
smooth vector field $V$ on $M$ such that $V\left( x\right) =v$ and $V\left(
y\right) \in T_{y}\mathcal{F}$ for all $y\in M$. If, in addition, there is a
Riemannian metric $g$ on $M$ such that every geodesic that is perpendicular
at one point to a leaf is perpendicular to every leaf it meets, we say that
the triple $\left( M,\mathcal{F},g\right) $ is a \textbf{singular Riemannian
foliation}.

If all the leaves of $\mathcal{F}$ have the same dimension, $\mathcal{F}$ is
called regular. In this case, the condition on $g$ given above is equivalent
to $g$ being a bundle-like metric.

Let the stratum $\Sigma _{r}\subseteq M$ denote the union of leaves of
dimension $r$. Then the restriction of $\mathcal{F}$ and $g$ to each $\Sigma
_{r}$ is a Riemannian foliation with bundle-like metric. The stratum
corresponding to leaves of the smallest dimension is a compact submanifold,
called the \textbf{minimal stratum}. The stratum corresponding to leaves of
maximal dimension is open and dense in $M$ and is called the \textbf{regular
stratum}. The closures of the leaves of a singular Riemannian foliation are
submanifolds, and the restriction of $\mathcal{F}$ to one of these leaf
closures is a [transversally locally homogeneous] regular Riemannian
foliation.

A \textbf{singular Riemannian flow} is a singular Riemannian foliation such
that the maximal dimension of each leaf is one.

We say that a smooth vector field $X$ on a smooth manifold $M$ is a \textbf{%
transverse Killing vector field} if there exists a Riemannian metric on $M$
such that the singular flow generated by $X$ is a singular Riemannian flow.
If the zero set $\Sigma $ of $X$ is nondegenerate, meaning the normal
Hessian of $X$ is invertible at $\Sigma $, we say that $X$ is a \textbf{%
nondegenerate transverse Killing vector field}. One can always construct a
nondegenerate transverse Killing vector field corresponding to any oriented
singular Riemannian flow. We remark that in other sources the term
\textquotedblleft transverse Killing\textquotedblright\ implies a choice of
metric on the normal bundle to the foliation, but we do not specify this
metric in our definition.

We start by establishing the structure of an oriented singular Riemannian
flow $\left( M,\mathcal{F},g\right) $ in the tubular neighborhood of a
component of the singular stratum $\Sigma :=\Sigma _{0}$ in Theorem \ref%
{TubularRiemFlowIsIsometric} and Corollary \ref%
{isometricModificationCorollary}. These theorems resemble slice theorems
(such as in \cite{MolPier}, \cite{AlexPlus}, \cite{MenRad}), but the new
results in this paper are stronger for flows in that they apply to the entire tubular
neighborhood of a singular stratum rather than to the neighborhood of a
singular leaf. We show that there exists a new metric $g^{\prime }$ on $M$
for which $\left( M,\mathcal{F},g^{\prime }\right) $ is a singular
Riemannian flow on $M$ that restricts to an isometric flow on the tubular
neighborhood. Note that every vector field that generates an isometric flow
for some metric on $M$ is automatically a nondegenerate transverse Killing
vector field and thus generates a singular Riemannian flow. It is easy to
construct transverse Killing fields that are not global Killing vector
fields for any metric; equivalently, there are singular Riemannian flows
that are not foliated-diffeomorphic to singular isometric flows. See
Examples \ref{RP_example} and \ref{projective_example}. In addition, a
nonorientable flow is considered in Example \ref{nonorientable_example}. In
Section 3, we establish some technical results which help to localize
computations to the tubular neighborhood of $\Sigma $.

\vspace{0in}The main result of the paper is Theorem \ref{mainTheorem}. In
this theorem, we provide the formula that computes characteristic numbers of
an even-dimensional, oriented closed manifold as the sum of residues at the components of the
zero set of a nondegenerate transverse Killing vector field that generates a 
singular Riemannian flow.
We prove that the Lie derivative of the field induces an isometric flow on the normal bundle of each 
component of the singular stratum, and the residue at this component is defined in
terms of the invariants of this action.
In the case when the
singular Riemannian flow is not orientable, the argument is easily handled
by Theorem \ref{nonorientableCase}. These theorems specialize to the results
in \cite{BaumCh} in the case when the singular Riemannian flow is in fact a
global isometric flow for some metric. One simple consequence, Corollary \ref%
{EulerCor}, is the formula for the Euler characteristic,%
\begin{equation*}
\chi \left( M\right) =\sum_{j}\chi \left( \Sigma _{j}\right) ,
\end{equation*}%
where $\Sigma _{j}$ are the components of the singular stratum of a possibly
nonorientable Riemannian flow $\left( M,\mathcal{F},g\right) $. This formula
was previously known when $\Sigma _{j}$ are the zero sets of a Killing
vector field; see \cite{Kob1}. See also Corollary \ref{SignatureCor} for a
new formula for the signature of a manifold endowed with a singular
Riemannian flow whose singular stratum is a finite set of points.

We now briefly discuss the history of this problem. In the celebrated paper 
\cite{bott1}, R. Bott showed how to compute the Pontryagin and other
characteristic numbers from isolated singular points of holomorphic vector
fields or of infinitessimal isometries. In \cite{bott2}, he generalized his
result in the holomorphic case to allow vector fields whose zero sets are
submanifolds. In \cite{APS3}, M. Atiyah and I. Singer used the $G$-signature
theorem, a special case of the index theorem, to give the formula for the
characteristic numbers of a singular isometric flow in terms of integrals of
characteristic forms over the singular stratum of the flow. In \cite{BaumCh}%
, P. Baum and J. Cheeger use purely differential-geometric and Stokes'
theorem techniques to derive the same result. One consequence of all these
results is that if there exists a nonvanishing Killing vector field on a
closed Riemannian manifold, then all of its characteristic numbers vanish.
In \cite{Carr1.5}, Y. Carri\`{e}re showed that any Riemannian manifold with
a nonsingular Riemannian flow has Gromov minimal volume zero, so as a
consequence all of the characteristic numbers of that manifold are zero,
consistent with Theorem~\ref{mainTheorem}. In \cite{Mei}, X. Mei considered
a singular Riemannian foliation and a variant of curvature coming from the
curvature of the normal bundle to the foliation. The author gave a formula
for the residue of a characteristic polynomial of this type of curvature at
a connected component of the singular stratum. These are not the same as the
residues used to compute the characteristic numbers of the manifold, which
are computed in this paper.

\section{Structure of the foliation near the singular stratum}

Let $\left( M,\mathcal{F},g\right) $ be an oriented singular Riemannian flow
on a smooth, connected, compact manifold $M$, and let $\Sigma $ be the
singular stratum of $\mathcal{F}$; note that each singular leaf is a point
in $\Sigma $. In this section, we show that the metric can be modified to
another metric such that the flow is still a singular Riemannian flow and
now its restriction to a tubular neighborhood of $\Sigma $ is an isometric
flow.

\subsection{The flow on a normal disk}

By \cite[Proposition 6.3]{Mo}, the singular stratum $\Sigma $ of the flow $%
\mathcal{F}$ is an embedded submanifold of $M$, which may have more than one
connected component. Let $\exp ^{\bot }:N\Sigma \rightarrow M$ be the normal
exponential map for $N\Sigma \subseteq \left. TM\right\vert _{\Sigma }$. For
any submanifold $V\subseteq M$, let $\mathrm{Tub}_{\varepsilon }\left(
V\right) $ denote the set of points of distance $<\varepsilon $ from $V$ in $%
M$. We also let $N\left( V,\varepsilon \right) =\left( \exp ^{\bot }\right)
^{-1}\left( \mathrm{Tub}_{\varepsilon }\left( V\right) \right) \subseteq NV$%
. Fix $x\in \Sigma $. Consider the normal disk $D_{\varepsilon }\left(
x\right) =\exp ^{\bot }\left( N_{x}\Sigma \right) \cap \mathrm{Tub}%
_{\varepsilon }\left( \Sigma \right) $. We will show in this section that $%
\mathcal{F}$ restricts to a Riemannian flow on $D_{\varepsilon }\left(
x\right) $.

Let $B_{\varepsilon }\left( x\right) \subseteq M$ denote the ball of radius $%
\varepsilon $ around $x$, and so $N\left( x,\varepsilon \right) $ is the
subset of $T_{x}M$ consisting of vectors of length $<\varepsilon $. There
exists $\varepsilon >0$ such that the exponential map $\exp :N\left(
x,\varepsilon \right) \rightarrow B_{\varepsilon }\left( x\right) $ is a
diffeomorphism. For $0<\lambda \leq 1$ the homothetic transformation $%
h_{\lambda }:B_{\varepsilon }\left( x\right) \rightarrow B_{\lambda
\varepsilon }\left( x\right) $ is defined by $h_{\lambda }\left( \exp \left(
v_{y}\right) \right) =\exp \left( \lambda v_{y}\right) $. By the
distance-preserving property, the foliation restricts to each sphere of
radius $r$ centered at $x$ for $0<r<\varepsilon $; let $\mathcal{F}%
_{B_{\varepsilon }}=\left. \mathcal{F}\right\vert _{B_{\varepsilon }\left(
x\right) }$.

\begin{lemma}
(Special case of the \textbf{homothetic transformation lemma}, \cite[Section
6.2]{Mo}) For sufficiently small $\varepsilon >0$, the homothetic
transformation $h_{\lambda }$ maps $\mathcal{F}_{B_{\varepsilon }}$ to $%
\mathcal{F}_{B_{\lambda \varepsilon }}$.
\end{lemma}

Let $S\left( \Sigma ,r\right) \overset{p_{r}}{\rightarrow }\Sigma $ denote
the fiber bundle of spheres of radius $r>0$ in $N\Sigma $. Let $S_{r}\Sigma
=\exp ^{\bot }\left( S\left( \Sigma ,r\right) \right) $. For sufficiently
small $r$, $S_{r}\Sigma $ is an embedded submanifold such that the
restriction $\left( S_{r}\Sigma ,\left. \mathcal{F}\right\vert _{S_{r}\Sigma
},\left. g\right\vert _{S_{r}\Sigma }\right) $ is a nonsingular Riemannian
flow (\cite[Section 6.2]{Mo}). Let $\pi _{r}=p_{r}\circ \left( \exp ^{\bot
}\right) ^{-1}:S_{r}\Sigma \rightarrow \Sigma $ be the sphere bundle
projection.

\begin{lemma}
For sufficiently small $r$ and each $x\in \Sigma $, the singular Riemannian
foliation $\mathcal{F}\ $restricts to the submanifold $\pi _{r}^{-1}\left(
x\right) $, and $\left. \mathcal{F}\right\vert _{\pi _{r}^{-1}\left(
x\right) }$ is a nonsingular Riemannian flow.\label{SphereLemma}

\begin{proof}
For $x\in \Sigma $, let $V_{r}\left( x\right) $ be the set of points of
distance $r$ from $x$ in $M$, where $r$ is sufficiently small. Then the
Riemannian flow $\left( M\setminus \Sigma ,\left. \mathcal{F}\right\vert
_{M\setminus \Sigma }\right) $ restricts to a Riemannian flow on $%
V_{r}\left( x\right) $, since the geodesics through $x$ must remain
orthogonal to every point of each fixed leaf of $\left( M\setminus \Sigma
,\left. \mathcal{F}\right\vert _{M\setminus \Sigma }\right) $. Also, by \cite%
[Sections 6.2 through 6.4]{Mo}, $\mathcal{F}$ restricts to a nonsingular
Riemannian flow on $S_{r}\Sigma $. Thus, the Riemannian flow must also
restrict to $V_{r}\left( x\right) \cap S_{r}\Sigma $, which is $\pi
_{r}^{-1}\left( x\right) $, because of the following argument. Clearly, $\pi
_{r}^{-1}\left( x\right) \subseteq V_{r}\left( x\right) $, since all of its
points are a distance $r$ from $x$, and likewise $\pi _{r}^{-1}\left(
x\right) \subseteq S_{r}\Sigma $ by construction. Since $\pi _{r}^{-1}\left(
x\right) $ and $V_{r}\left( x\right) \cap S_{r}\Sigma $ are smooth spheres
and of the same dimension for small $r$, we have $V_{r}\left( x\right) \cap
S_{r}\Sigma =\pi _{r}^{-1}\left( x\right) $.
\end{proof}
\end{lemma}

\subsection{Infinitesimal flow on the normal bundle to the singular stratum}

We will replace the metric on $\mathrm{Tub}_{\varepsilon }\left( \Sigma
\right) $ with a linearized metric, as follows. Let $g^{\Sigma }$ be the
original metric restricted to $T\Sigma $, and let $\widetilde{g^{\Sigma }}$
be the pullback of $g^{\Sigma }$ to the horizontal space $\mathcal{H}\subset
T\left( N\left( \Sigma ,\varepsilon \right) \right) \subset T\left( N\Sigma
\right) $ given by the normal Levi-Civita connection. Let $g^{\bot }$ be the
metric on $N_{y}\Sigma $ for each $y\in \Sigma $, and let $\widetilde{%
g^{\bot }}$ be the corresponding translation-invariant metric on the fibers
of $N\Sigma \rightarrow \Sigma $. We define the metric on the total space of 
$N\left( \Sigma ,\varepsilon \right) $ as $\widetilde{g^{\Sigma }}\oplus 
\widetilde{g^{\bot }}$. We may then transplant this metric to $\mathrm{Tub}%
_{\varepsilon }\left( U\right) $ via $\left( \exp ^{\bot }\right) _{\ast }$.
We call the resulting metric the \textbf{linearized metric} $g_{L}$ on $%
\mathrm{Tub}_{\varepsilon }\left( U\right) $. By the results of \cite[%
Section 6.4, applied to $\Sigma $]{Mo}, $\left( \mathrm{Tub}_{\varepsilon
}\left( U\right) ,\left. \mathcal{F}\right\vert _{\mathrm{Tub}_{\varepsilon
}\left( U\right) },g_{L}\right) $ is a singular Riemannian foliation, as is $%
\left( N\left( \Sigma ,\varepsilon \right) ,\left( \exp ^{\bot }\right)
^{-1}\left( \left. \mathcal{F}\right\vert _{\mathrm{Tub}_{\varepsilon
}\left( U\right) }\right) ,\widetilde{g^{\Sigma }}\oplus \widetilde{g^{\bot }%
}\right) .$ We remark that since the foliation is $1$-dimensional away from $%
\Sigma $, the exponential map takes the leaves of the infinitesimal
foliation onto the leaves of $\mathcal{F}$. Further, the foliation is
invariant under homothetic transformations with respect to the stratum $%
\Sigma $.

\begin{lemma}
\label{isometricDiskLemma}For each $y\in \Sigma $, the restriction of $%
\mathcal{F}$ to $\exp ^{\bot }\left( N_{y}\Sigma \cap N\left( \Sigma
,\varepsilon \right) \right) $ is an isometric flow in the linearized
metric. Similarly, the restriction of $\left( \exp ^{\bot }\right)
^{-1}\left( \left. \mathcal{F}\right\vert _{\mathrm{Tub}_{\varepsilon
}\left( U\right) }\right) $ to $N_{y}\Sigma \cap N\left( \Sigma ,\varepsilon
\right) $ is also an isometric flow.
\end{lemma}

\begin{proof}
For fixed $y\in \Sigma $, by Lemma \ref{SphereLemma}, the restriction of $%
\mathcal{F}$ to $D=\exp ^{\bot }\left( N_{y}\Sigma \cap N\left( \Sigma
,\varepsilon \right) \right) $ is a Riemannian flow (for the original
metric) on a metric ball of radius $\varepsilon $ and centered at $y$, so
that the spheres at each radius $r$ with $0<r<\varepsilon $ are foliated by $%
\mathcal{F}$, and such that the homothetic transformation $r\mapsto
r^{\prime }$ preserves the foliation (\cite[Section 6.2]{Mo}). The foliation
induces a foliation of the unit sphere $S_{y}D$ in the tangent space $T_{y}D$%
, where $v,w$ are on the same leaf if and only if $\exp _{y}\left( tv\right) 
$ and $\exp _{y}\left( tw\right) $ are on the same leaf for a fixed $t\in
\left( 0,\varepsilon \right) $ [and hence for all $t\in \left( 0,\varepsilon
\right) $ by the homothetic transformation property]. We claim that the
induced foliation $\mathcal{F}_{S}$ on $S_{y}D$ with the Euclidean metric is
a Riemannian flow. If not, then there exist two nearby leaves of a foliation
chart of $\mathcal{F}_{S}$ that are not locally equidistant in the Euclidean
metric on $S_{y}D$. Then it follows that the corresponding leaves of $%
\mathcal{F}$ close to the origin $y$ are also not locally equidistant on a
foliation chart of the metric sphere of radius $t$ for small $t$. This is a
contradiction. Then, since $\mathcal{F}_{S}$ is a Riemannian flow on the
round sphere $S_{y}D$, by \cite{GrGr}, the foliation $\mathcal{F}_{S}$ is a
foliation arising from an isometric flow.
\end{proof}

\subsection{The structure of the tubular neighborhood\label%
{StrTubNbhdSection}}

The restriction of the normal exponential map $\exp ^{\bot }:N\left( \Sigma
,\varepsilon \right) \rightarrow \mathrm{Tub}_{\varepsilon }\left( \Sigma
\right) $ is a diffeomorphism. We have seen in the previous section and in
Lemma \ref{isometricDiskLemma} that the oriented singular Riemannian flow is
conjugate through $\exp ^{\bot }$ to a singular Riemannian foliation on $%
N\left( \Sigma ,\varepsilon \right) $ such that the restriction of the flow
to each fiber $N_{y}\Sigma \cap N\left( \Sigma ,\varepsilon \right) $ is a
linear isometric flow whose only fixed point is the origin. Thus the normal
bundle has even rank.

More specifically (see \cite{GrGr}), each such flow on a single fiber $%
N_{y}\Sigma $ of the normal bundle is conjugate to a linear isometric flow
on $\mathbb{C}^{k}$ with no fixed points other than the origin, where $2k$
is dimension of $N_{y}\Sigma $. Thus, by the representation theory of
isometric actions, there exists an isomorphism $N_{y}\Sigma \cong \mathbb{C}%
^{k}$ such that the isometric flow $\gamma :\mathbb{R}\rightarrow U\left( k,%
\mathbb{C}\right) $ has the form 
\begin{equation}
\gamma \left( t\right) =\left( 
\begin{array}{cccc}
\exp \left( i\alpha _{1}t\right) & 0 & 0 & 0 \\ 
0 & \exp \left( i\alpha _{2}t\right) & \vdots & \vdots \\ 
\vdots & 0 & \ddots & 0 \\ 
0 & ... & 0 & \exp \left( i\alpha _{k}t\right)%
\end{array}%
\right)  \label{formOfIsometricFlow}
\end{equation}%
acting on $z\in \mathbb{C}^{k}$, where $\alpha _{1}\leq \alpha _{2}\leq
...\leq \alpha _{k}$ are nonzero real constants. We select the constants so
that $\alpha _{1}^{2}+...+\alpha _{k}^{2}=1$. A family of such flows, in
other words the flow on the total space $N\Sigma $ dependent on the base
point $y\in \Sigma $, may move the planes $P_{i}=\left\{ z:z_{j}=\text{%
constant for }j\neq i\right\} $ but we claim must only multiply the $\left(
\alpha _{1},\alpha \,_{2},...,\alpha _{k}\right) $ by a scalar. Because of
the normalization, the scalar must be $\pm 1$, so the foliation may only
change its orientation. We now prove this claim in the following lemma,
proposition, and corollary.

\begin{lemma}
Let $\gamma :\mathbb{R}\rightarrow U\left( k,\mathbb{C}\right) $ be a family
as above, which depends on the choice of $\left( \alpha _{1},...,\alpha
_{k}\right) \in S^{k-1}$. Let $C\left( z\right) =C\left(
z_{1},...,z_{k}\right) $ denote the real dimension of the closure of the set 
$\left\{ \gamma \left( t\right) z:t\in \mathbb{R}\right\} $ in $%
S^{2k-1}\subseteq \mathbb{C}^{k}$. Then the maximum value of $C\left(
z\right) $ for $z$ in any open set is $\dim _{\mathbb{Q}}\left\{ \alpha
_{1},...,\alpha _{k}\right\} =k-\dim _{\mathbb{Q}}\left\{ \beta \in \mathbb{Q%
}^{k}:\beta \cdot \alpha =0\right\} $.

\begin{proof}
This is well--known and generalizes the Kronecker foliation.
\end{proof}
\end{lemma}

\begin{proposition}
Let $I$ be any closed interval of positive length, and let $\alpha
:I\rightarrow S^{k-1}$, with $\alpha \left( s\right) =\left( \alpha
_{1}\left( s\right) ,...,\alpha _{k}\left( s\right) \right) $, be a
continuous, nonconstant path. Define $D_{\alpha }:I\rightarrow \mathbb{Z}$
by $D_{\alpha }\left( s\right) =\dim _{\mathbb{Q}}\left\{ \beta \in \mathbb{Q%
}^{k}:\beta \cdot \alpha \left( s\right) =0\right\} $ for $s\in I$. Then $%
D_{\alpha }$ is not constant.
\end{proposition}

\begin{proof}
We induct on $k$.\footnote{%
The idea of proof was suggested to us by G. Gilbert.} If $k=1$, the
statement is true vacuously. Now, suppose that for some $i\geq 1$, the
statement is true for $1\leq k\leq i$. Now let $k=i+1$, so that $\alpha
:I\rightarrow S^{i}$ is a continuous, nonconstant path. Let $\left[
a_{0},b_{0}\right] =I$, and we enumerate $\mathbb{Q}^{i+1}\setminus \left\{
0\right\} =\left\{ \beta _{1},\beta _{2},...\right\} $. We construct a
nested sequence of intervals $\left[ a_{n},b_{n}\right] _{n\geq 1}$ such
that $0\leq a_{n}<b_{n}\leq 1$, $\left. \alpha \right\vert _{\left[
a_{n},b_{n}\right] }$ is not constant, and $\beta _{j}\cdot \alpha \left(
s\right) \neq 0$ for all $s\in \left[ a_{n},b_{n}\right] $ and all $1\leq
j\leq n$. To construct $\left[ a_{n},b_{n}\right] $ from $\left[
a_{n-1},b_{n-1}\right] $, let $U_{n}=\left\{ s\in \left[ a_{n-1},b_{n-1}%
\right] :\beta _{n}\cdot \alpha \left( s\right) \neq 0\right\} $. The set $%
U_{n}$ is open in $\left[ a_{n-1},b_{n-1}\right] $. If $U_{n}=\left[
a_{n-1},b_{n-1}\right] $, we let $\left[ a_{n},b_{n}\right] =\left[
a_{n-1},b_{n-1}\right] $ and continue. If $U_{n}=\emptyset $, then $\left.
\alpha \right\vert _{\left[ a_{n-1},b_{n-1}\right] }$ is a nonconstant map
to the lower-dimensional sphere $S^{\prime }=S^{i}\cap \left\{ x\in \mathbb{R%
}^{i+1}:\beta _{n}\cdot x=0\right\} $, and by the induction hypothesis $%
D_{\alpha }$ is not constant (because it is $1+D_{\alpha }^{\prime }$ where $%
D_{\alpha }^{\prime }$ is the corresponding dimension on $S^{\prime }$), so
the statement is true for $k=i+1$, and we stop. Otherwise, $U_{n}$ is a
proper, nonempty open subset of $\left[ a_{n-1},b_{n-1}\right] $. Then $%
\beta _{n}\cdot \alpha \left( s\right) $ is nonconstant (since it is nonzero
on $U_{n}$ and $0$ on the complement) on every connected open subset of $%
U_{n}$, making $\alpha $ nonconstant on any connected open subset of $U_{n}$%
, i.e. on an open interval $I_{n}\subseteq U_{n}$. We then let $\left[
a_{n},b_{n}\right] $ be a closed interval within $I_{n}$, sufficiently large
so that $\alpha $ is not constant on $\left[ a_{n},b_{n}\right] $. If this
process of constructing $\left[ a_{n},b_{n}\right] $ terminates, we are
done, as mentioned before; otherwise we choose an $x\in
\bigcap\limits_{n\geq 0}\left[ a_{n},b_{n}\right] $. Since $\beta _{j}\cdot
\alpha \left( x\right) \neq 0$ for all $j$, $D_{\alpha }\left( x\right) =0$.
Thus, if $D_{\alpha }\left( s\right) \neq 0$ for any $s\in I$, we have
proved the case $k=i+1$. Otherwise, we have that $D_{\alpha }\left( s\right)
=0$ for all $s\in I$. However, this is not possible, because of the
following. Since $\alpha $ is not constant and continuous, there exist $%
s_{1},s_{2}\in I$ such that $\alpha \left( s_{1}\right) $ and $\alpha \left(
s_{2}\right) $ are not on the same line. But then, there exists $\beta \in 
\mathbb{Q}^{i+1}\setminus \left\{ 0\right\} $ contained in the intersection
of the open half spaces $\beta \cdot \alpha \left( s_{1}\right) <0$ and $%
0<\beta \cdot \alpha \left( s_{2}\right) $. By the intermediate value
theorem, there exists $s_{3}$ between $s_{1}$ and $s_{2}$ such that $\beta
\cdot \alpha \left( s_{3}\right) =0$, so $D_{\alpha }\left( s_{3}\right) >0$%
, a contradiction. Thus, in every situation the case $k=i+1$ is true, so by
induction the statement holds for all $k$.
\end{proof}

\begin{corollary}
(Rigidity of Isometric flows on spheres)\label{RIgidityIsomFlowsCorollary}
For some connected open set $U$ of a manifold $\Sigma $, suppose that $%
\mathbb{C}^{k}\times U$ contains an oriented singular Riemannian flow with
associated bundle-like metric such that the flow restricts to a singular
isometric flow on each $\mathbb{C}^{k}\times \left\{ s\right\} $ of the form 
$\left( z_{1},z_{2},...,z_{k}\right) \mapsto \left( \exp \left( i\alpha
_{1}\left( s\right) t\right) z_{1},\exp \left( i\alpha _{2}\left( s\right)
t\right) z_{2},...,\exp \left( i\alpha _{k}\left( s\right) t\right)
z_{k}\right) $ for some $\alpha \left( s\right) =\left( \alpha _{1}\left(
s\right) ,...,\alpha _{k}\left( s\right) \right) \in S^{k-1}$. Then $\alpha $
is constant on $U$.

\begin{proof}
By the previous two propositions, if $\alpha $ is not constant on $U$, then
there is a one-parameter family $\alpha \left( s\right) $ with $s\in U$ such
that the top dimensions of the leaf closures of the flow restricted to
spheres will change discontinuously. In such a situation, it is not possible
that the distance between leaves is locally constant, so that the singular
flow cannot be Riemannian.
\end{proof}
\end{corollary}

As a result, the restriction of our foliation to a ball in $N\Sigma $ over
each point of a connected component of $\Sigma $ must be the set of orbits
of the group action 
\begin{equation*}
\left( z_{1},z_{2},...,z_{k}\right) \mapsto \left( \exp \left( i\alpha
_{1}\left( s\right) t\right) z_{1},\exp \left( i\alpha _{2}\left( s\right)
t\right) z_{2},...,\exp \left( i\alpha _{k}\left( s\right) t\right)
z_{k}\right) ,
\end{equation*}
where the subspaces corresponding to the span of all $z_{j}$ for equal $%
\alpha _{j}$ may vary with $y\in \Sigma $. In order to determine the
allowable variations in these subspaces of the normal spaces $N_{y}\Sigma $,
we first must understand all orientation-preserving isometries of $\mathbb{C}%
^{k}$ fixing the origin $y$ and preserving the foliation. In particular, the
isometry $L=L_{y}$ at $y$ must preserve the tangent bundle of the foliation.
This means that the vector field is determined by the matrix multiplication:%
\begin{equation*}
V_{x}=Mx=\left( 
\begin{array}{ccccccc}
0 & -\alpha _{1} & 0 & 0 & 0 & ... & 0 \\ 
\alpha _{1} & 0 & 0 & 0 & \vdots & \vdots & \vdots \\ 
0 & 0 & 0 & -\alpha _{2} & 0 & \vdots & \vdots \\ 
0 & 0 & \alpha _{2} & 0 & 0 & 0 & \vdots \\ 
0 & \vdots & 0 & 0 & \ddots & 0 & 0 \\ 
\vdots & \vdots & \vdots & \vdots & 0 & 0 & -\alpha _{k} \\ 
0 & ... & ... & ... & 0 & \alpha _{k} & 0%
\end{array}%
\right) x=\left( 
\begin{array}{cccc}
R_{1} & 0 & 0 & 0 \\ 
0 & R_{2} & 0 & 0 \\ 
0 & 0 & \ddots & 0 \\ 
0 & 0 & 0 & R_{k}%
\end{array}%
\right) x
\end{equation*}%
must be preserved up to a positive scalar multiple. This means that the push
forward of the isometry maps $V$ to a positive scalar multiple of itself.
That is, we require that%
\begin{equation*}
LMx=\lambda M\left( Lx\right)
\end{equation*}%
for some $\lambda >0,$ for all $x$. Since $M$ and $L$ are invertible, taking
determinants shows that $\lambda =1$.

\vspace{0in}We renumber the $\left\{ \alpha _{j}\right\} $ to distinct $%
\beta _{j}$ with multiplicities $\mu _{j}$. A simple linear algebra argument
yields the following. Using the standard inclusion of $GL_{k}\left( \mathbb{C%
}\right) $ in $GL_{2k}\left( \mathbb{R}\right) $ , we have 
\begin{equation*}
\left\{ L\in GL_{2k}\left( \mathbb{R}\right) :LM=ML\right\} =\left\{
L=\left( 
\begin{array}{cccc}
B_{1} & 0 & 0 & 0 \\ 
0 & B_{2} & 0 & 0 \\ 
0 & 0 & \ddots & 0 \\ 
0 & 0 & 0 & B_{p}%
\end{array}%
\right) \in GL_{k}\left( \mathbb{C}\right) :B_{j}\in GL_{\mu _{j}}\left( 
\mathbb{C}\right) \right\} .
\end{equation*}

This implies that the tubular neighborhood of each connected component of
the singular stratum of the Riemannian flow is diffeomorphic to a
neighborhood of the zero section in a direct sum of complex vector bundles $%
E_{1},...,E_{p}$ over the singular stratum. The flow is multiplication by $%
\exp \left( i\alpha _{j}t\right) $ on each $E_{j}$, where $\left( \alpha
_{1},...,\alpha _{p}\right) $ is a locally constant vector. Considering the
torus action by $\left( z_{1},...,z_{p}\right) \in T^{p}$ on the vector
bundle by 
\begin{equation*}
\left( z_{1},...,z_{p}\right) \left( v_{1},...,v_{p}\right) \mapsto \left(
z_{1}v_{1},...,z_{p}v_{p}\right) ,
\end{equation*}
the flow is a subgroup of this torus action. We average the metric over the
torus flow to get a new metric $g_{1}$, so that the Riemannian flow is an
isometric flow near each connected component of the singular stratum for
this new metric. We summarize this discussion in the following theorem.

\begin{theorem}
\label{TubularRiemFlowIsIsometric}The restriction of the oriented singular
Riemannian flow $\left( M,\mathcal{F}\right) $ to a tubular neighborhood of
a connected component of the singular stratum is foliated-diffeomorphic to
an isometric flow on the same tubular neighborhood.
\end{theorem}

\begin{lemma}
\label{patchingLemma}Let $\left( U,\mathcal{F}\right) $ be a smooth,
oriented flow on a smooth manifold $U$, and suppose that there exist two
bundle-like metrics $g$, $\widetilde{g}$ for the flow. Let $\psi
:U\rightarrow \left[ 0,1\right] $ be a basic function. Then there exists a
smooth bundle-like metric $g_{\psi }$ for $\left( U,\mathcal{F}\right) $
such that $g_{\psi }=g$ on $\psi ^{-1}\left( 0\right) $ and $g_{\psi }=%
\widetilde{g}$ on $\psi ^{-1}\left( 1\right) $.

\begin{proof}
Near any point we may choose a foliation chart with adapted coordinates $%
(x,y)\in \mathbb{R}\times \mathbb{R}^{q}$. A metric on this chart is a
bundle-like metric if and only if it has the form%
\begin{equation*}
\theta \otimes \theta +\sum_{\alpha ,\beta =1}^{q}h_{\alpha \beta
}~dy^{\alpha }\otimes dy^{\beta },
\end{equation*}%
where $\theta $ is a one-form such that $\ \theta \left( \partial
_{x}\right) >0$ and $h_{\alpha \beta }$ is a positive definite symmetric
matrix of basic functions; see \cite[Section IV, Proposition 4.2]{Re1}. In
fact, $\theta $ is always dual to a unit vector field tangent to the flow
and $T\mathcal{F}^{\bot }=\ker \theta $; thus, up to a sign $\theta $ is
defined globally.\newline
We express $g$ in this form as 
\begin{equation*}
g=\theta \otimes \theta +\sum_{\alpha ,\beta =1}^{q}h_{\alpha \beta }\left(
y\right) ~dy^{\alpha }\otimes dy^{\beta }=g_{T}+g_{N},
\end{equation*}%
Similarly, we have 
\begin{equation*}
\widetilde{g}=\widetilde{\theta }\otimes \widetilde{\theta }+\sum_{\alpha
,\beta =1}^{q}\widetilde{h}_{\alpha \beta }(y)~dy^{\alpha }\otimes dy^{\beta
}=\widetilde{g}_{T}+\widetilde{g}_{N}.
\end{equation*}%
Let $\Pi :TU\rightarrow T\mathcal{F}$ be the orthogonal projection defined
by the first metric $g$. Since $\widetilde{\theta }\left( \partial
_{x}\right) >0$, also $\left( \Pi ^{\ast }\widetilde{\theta }\right) \left(
\partial _{x}\right) >0$. We define a new bundle-like metric $\overline{g}$
by 
\begin{equation*}
\overline{g}=\left( \Pi ^{\ast }\widetilde{\theta }\right) \otimes \left(
\Pi ^{\ast }\widetilde{\theta }\right) +\sum_{\alpha ,\beta =1}^{q}%
\widetilde{h}_{\alpha \beta }(y)~dy^{\alpha }\otimes dy^{\beta }=\overline{g}%
_{T}+\widetilde{g}_{N}.
\end{equation*}%
Note that the bundles $T\mathcal{F}$ and $T\mathcal{F}^{\perp }$ agree for
both $g$ and $\overline{g}$, since $\ker \theta =\ker \left( \Pi ^{\ast }%
\widetilde{\theta }\right) $. Since $\widetilde{\theta }$ and $\Pi ^{\ast }%
\widetilde{\theta }$ are globally defined up to a sign, it is clear that $%
\widetilde{g}$ and $\overline{g}$ are homotopic through a global homotopy
transforming $\widetilde{\theta }$ to $\Pi ^{\ast }\widetilde{\theta }$.
Next, we may homotop $\overline{g}$ to $g$ by separately homotoping $%
\overline{g}_{T}$ to $g_{T}$ on $T\mathcal{F}$ and $\widetilde{g}_{N}$ and $%
g_{N}$ on $T\mathcal{F}^{\bot }$. With a simple concatenation and
adjustments to the parameter to make the concatenation smooth, a smooth
homotopy connecting $g$ to $\widetilde{g}$ can be formed, and this homotopy
is a homotopy of bundle-like metrics. If we replace the homotopy parameter
with the single basic function $\psi $, it is easy to check that the
resulting metric $g_{\psi }$ is bundle-like for $\left( U,\mathcal{F}\right) 
$.
\end{proof}
\end{lemma}

\begin{corollary}
\label{isometricModificationCorollary}Given a nondegenerate transverse
Killing field $X$ on a compact, smooth manifold $M$, there exists a metric
on $M$ for which the restriction of $X$ to a tubular neighborhood of the
singular stratum $\Sigma $ is a Killing vector field, and such that each
component of $\Sigma $ is a totally geodesic submanifold.
\end{corollary}

\begin{proof}
Use the Lemma to patch the original metric $g$ with the tubular neighborhood
metric $\widetilde{g}$ in the Theorem above using a basic cutoff function $%
\rho $ that is $1$ in a neighborhood of $\Sigma $ and is $0$ outside a
larger neighborhood; the new metric $g_{\rho }$ does the job.
\end{proof}

\section{Modification of the metric and localization}

It is enough to consider the case where the dimension of our manifold $M$ is 
$2m$, since the characteristic numbers vanish on odd-dimensional manifolds.

\subsection{Estimate on the complement of a tubular neighborhood of the
singular stratum\label{FirstModificationSubsection}}

Our manifold $M$ is endowed with a oriented singular Riemannian flow $%
\mathcal{F}$ whose tangent bundle is given by the span of the vector field $%
X $. As in Corollary \ref{isometricModificationCorollary}, we choose a
metric $g$ on $M$ and an $\varepsilon >0$ such that the restriction of $X$
to the tubular neighborhood $T_{\varepsilon }\Sigma $ is a Killing vector
field and that $\mathcal{F}$ is a singular Riemannian flow globally with $g$
bundle-like. We now modify this metric to a new metric $g_{t}$ as follows.
For $t>0$, let $g_{t}$ denote the metric on $M$ defined by $g_{t}\left(
X,Y\right) =g\left( X,Y\right) $ for all vectors $X$, $Y$ as long as one of
them is in $\left( T\mathcal{F}\right) ^{\bot }$ and $g_{t}\left( X,X\right)
=\rho \left( t,r\right) g\left( X,X\right) $ if $X\in T\mathcal{F}$. Here, $%
\rho \left( t,r\right) $ is a function on $M$ defined as $\rho \left(
t,r\right) =f\left( r\right) t^{2}+1-f\left( r\right) $ where $r=\mathrm{dist%
}\left( \cdot ,\Sigma \right) $ and 
\begin{equation*}
f\left( r\right) =\left\{ 
\begin{array}{ll}
0 & \text{if }r\leq \frac{\varepsilon }{3} \\ 
1 & \text{if }r\geq \frac{2\varepsilon }{3} \\ 
\text{smooth, increasing~} & \text{if }\frac{\varepsilon }{3}<r<\frac{%
2\varepsilon }{3}.%
\end{array}%
\right.
\end{equation*}

\begin{lemma}
\vspace{0in}For every $t>0$, the singular foliation $\left( M,\mathcal{F}%
,g_{t}\right) $ is Riemannian and restricts to an isometric flow on $%
T_{\varepsilon }\Sigma $.

\begin{proof}
First we consider the restriction of $\mathcal{F}$ to the tube $%
T_{\varepsilon }\Sigma $, where $X$ is a Killing field for $g$. Note that
this is still the tube of radius $\varepsilon $ in the new metric $g_{t}$.
The flow of $X$ preserves the tangent space $T\mathcal{F}$ and normal space $%
N\mathcal{F}$ and the metric on $N\mathcal{F}$. Since $X\rho =0$, the
Leibniz rule implies that the Lie derivative $\mathcal{L}_{X}g_{t}^{\text{%
\textrm{tan}}}$of the tangential metric is still zero. Then it follows that $%
\mathcal{L}_{X}g_{t}=0$. Next, outside of $T_{\varepsilon }\Sigma $, the
tangential metric is multiplied by the scalar $t^{2}$, so the leaves of the
foliation remain equidistant, so that the metric $g_{t}$ is still
bundle-like outside of $T_{\varepsilon }\Sigma $.
\end{proof}
\end{lemma}

\begin{proposition}
\label{OutsideTubeGoesToZero}Let $g_{t}$ be the family of Riemannian metrics
on the manifold $M$. Let $\omega _{t}$ be a characteristic form, an $\mathrm{%
ad}\left( \mathrm{SO}\left( 2m\right) \right) $-invariant polynomial in the
curvature $\left( M,g_{t}\right) $. Then 
\begin{equation*}
\int_{M\setminus T_{\varepsilon }\Sigma }\omega _{t}\rightarrow 0
\end{equation*}%
as $t\rightarrow 0$.

\begin{proof}
\vspace{0in}Similar to the Carrier\`{e} argument in \cite[Lemmas 1.1 - 1.2]%
{Carr1.5}, the sectional curvatures of $(M,g_{t})$ on $M\setminus
T_{\varepsilon }\Sigma $ remain bounded as $t\rightarrow 0$. The $2m$-degree
part of $\omega _{t}$ is a polynomial with coefficients independent of $t$
in the sectional curvatures of $\left( M,g_{t}\right) $ times the volume
form of $g_{t}$. Since the volume form of $g_{t}$ is $t$ times the volume
form of $g$, the result follows.
\end{proof}
\end{proposition}

\subsection{Estimate on the boundary of the tubular neighborhood}

\vspace{0in}We restrict our attention to $\overline{T_{\varepsilon }\Sigma }$%
, on which $X$ is a Killing field for the metric $g_{t}$.

We let 
\begin{equation*}
\omega _{t}=\phi \left( K_{t},...,K_{t}\right) ,
\end{equation*}%
where $\phi $ is a polynomial that is homogeneous of degree $m$ and $K_{t}$
is the Riemannian curvature two-form. Following \cite[Lemma 2]{bott1}, on $%
T_{\varepsilon }\Sigma \setminus \Sigma $, there exists a $2m-1$ form $\eta
_{t}$ defined as%
\begin{equation*}
\eta _{t}=\alpha \left\{ \phi _{K,t}^{1}+\phi _{K,t}^{2}d\alpha +...+\phi
_{K,t}^{m}\left( d\alpha \right) ^{m-1}\right\} .
\end{equation*}%
Here, $\alpha $ is the one-form defined as%
\begin{equation*}
\alpha \left( Y\right) =\frac{\left\langle X,Y\right\rangle }{\left\langle
X,X\right\rangle }=\frac{\left\langle X,Y\right\rangle _{t}}{\left\langle
X,X\right\rangle _{t}},
\end{equation*}%
and for $1\leq j\leq m$,%
\begin{equation*}
\phi _{K,t}^{j}=\binom{m}{j}\phi (\underset{j}{\underbrace{L_{t},...,L_{t}}}%
,K_{t},...,K_{t}),
\end{equation*}%
where $L_{t}$ is the endomorphism of $TM$ defined by%
\begin{equation*}
L_{t}=\mathcal{L}_{X}-\nabla _{X}^{g_{t}}.
\end{equation*}%
From \cite[Lemma 2]{bott1}, 
\begin{equation*}
i\left( X\right) \left( \omega _{t}-d\eta _{t}\right) =0,
\end{equation*}%
which then implies%
\begin{equation}
\omega _{t}-d\eta _{t}=0  \label{omegaIsDeta}
\end{equation}%
since $i\left( X\right) $ injective on top-degree forms on $T_{\varepsilon
}\Sigma \setminus \Sigma $.

\begin{proposition}
\label{etaEpsilonGoesToZero}With $\eta _{t}$ and $\varepsilon $ as above,%
\begin{equation*}
\int_{\partial T_{\varepsilon }\Sigma }\eta _{t}\rightarrow 0
\end{equation*}%
as $t\rightarrow 0$.

\begin{proof}
We have $\dim \partial T_{\varepsilon }\Sigma =2m-1$. Let $i:\partial
T_{\varepsilon }\Sigma \rightarrow M$ be the inclusion, so that the form $%
i^{\ast }\eta _{t}$ is well-defined, and%
\begin{equation*}
\int_{\partial T_{\varepsilon }\Sigma }\eta _{t}:=\int_{\partial
T_{\varepsilon }\Sigma }i^{\ast }\eta _{t}.
\end{equation*}%
Observe that $i^{\ast }\eta _{t}$ is a polynomial in sectional curvatures
and Christoffel symbols times the volume form of $\partial T_{\varepsilon
}\Sigma $. From Carrier\`{e}'s proof in \cite[Lemmas 1.1 - 1.2]{Carr1.5},
all the sectional curvatures and Christoffel symbols remain bounded as $%
t\rightarrow 0$, and the volume form of $i^{\ast }g_{t}$ is by construction $%
t$ times the volume form $i^{\ast }g$, since $X$ is tangent to $\partial
T_{\varepsilon }\Sigma $. The result follows.
\end{proof}
\end{proposition}

\section{Computation of the characteristic numbers}

Let $M$ be a compact, oriented, Riemannian manifold of dimension $2m$,
endowed with a singular Riemannian flow $\mathcal{F}$ corresponding to the
given transverse Killing vector field $X$. Let $FM\overset{\pi }{\rightarrow 
}M$ denote the oriented orthonormal frame bundle of $M$, and let $K$ be the $%
\mathfrak{o}\left( 2m\right) $-valued curvature two-form on $FM$. Let $\phi $
be an $\mathrm{ad}\left( \mathrm{SO}\left( 2m\right) \right) $-invariant
symmetric form of degree $m$ on $\mathfrak{o}\left( 2m\right) $. Then the
function $\phi \left( K\right) :=\phi \left( K,...,K\right) $ is the
pullback of the closed form on $M$, which by abuse of notation we also
denote $\phi \left( K\right) $. The number $\int_{M}{\phi \left( K\right) }$
is the characteristic number associated to $\phi $: 
\begin{equation*}
C_{\phi }:=\int_{M}{\phi \left( K\right) },
\end{equation*}%
which is independent of the metric (and curvature) on $M$. In particular, we
will use $K_{t}$, the Riemannian curvature two-form with respect to the
metric $g_{t}$ defined in Section \ref{FirstModificationSubsection}. Let 
\begin{equation*}
\omega _{t}=\phi \left( K_{t},...,K_{t}\right) .
\end{equation*}%
As in the previous sections, we fix an appropriate $\varepsilon >0$ and let $%
T_{\varepsilon }\Sigma $ denote the tubular neighborhood of the singular
stratum $\Sigma $ of $\mathcal{F}$, on which $X$ generates a $g_{t}$%
-isometric flow. Then, using Stoke's Theorem and $\omega _{t}=d\eta _{t}$ in 
$T_{\varepsilon }\Sigma $ from (\ref{omegaIsDeta}), 
\begin{eqnarray*}
C_{\phi } &=&\int_{M\setminus T_{\varepsilon }\Sigma }\omega
_{t}+\int_{T_{\varepsilon }\Sigma }\omega _{t} \\
&=&\int_{M\setminus T_{\varepsilon }\Sigma }\omega _{t}+\lim_{\delta
\rightarrow 0}\int_{T_{\varepsilon }\Sigma -T_{\delta }\Sigma }\omega _{t} \\
&=&\int_{M\setminus T_{\varepsilon }\Sigma }\omega _{t}+\int_{\partial
T_{\varepsilon }\Sigma }\eta _{t}-\lim_{\delta \rightarrow 0}\int_{\partial
T_{\delta }\Sigma }\eta _{t}
\end{eqnarray*}%
As $t\rightarrow 0$, by Proposition \ref{OutsideTubeGoesToZero}, Proposition %
\ref{etaEpsilonGoesToZero}, we obtain the following.

\begin{lemma}
\label{LocalizationLemma}With the notation above, the characteristic number
associated to $\phi $ satisfies 
\begin{equation*}
C_{\phi }=-\lim_{\delta \rightarrow 0}\int_{\partial T_{\delta }\Sigma }\eta
_{t}=-\lim_{\delta \rightarrow 0}\int_{\partial T_{\delta }\Sigma }\eta _{0},
\end{equation*}%
where for small $\delta $, the restriction of $\eta _{t}$ to $\partial
T_{\delta }\Sigma $ does not depend on $t$ and corresponds to the fixed
metric $g$.
\end{lemma}

\vspace{1pt}We now introduce the notation of the main theorem, much of which
is similar to that in \cite[Section 1]{BaumCh}. Given any invertible linear
transformation $A\in \mathfrak{o}\left( 2s\right) $, there exists an
orthonormal basis $\left\{ e_{1},...,e_{2s}\right\} $ for $\mathbb{R}^{2s}$
such that $Ae_{2j-1}=\lambda _{j}e_{2j}$ and $Ae_{2j}=-\lambda _{j}e_{2j-1}$
and $\lambda _{j}\geq 0$ for each $j$. The numbers $\lambda _{j}$ are called 
\textbf{skeigen-values}. It is well-known that if $\psi $ is an $\mathrm{ad}%
\left( \mathrm{SO}\left( 2s\right) \right) $-invariant symmetric
complex-valued polynomial on $\mathfrak{o}\left( 2s\right) $, there exists a
unique polynomial $\widehat{\psi }:\mathbb{R}^{s+1}\rightarrow \mathbb{C}$
such that%
\begin{equation*}
\psi \left( A\right) =\widehat{\psi }\left( \lambda _{1},...,\lambda
_{s}\right) .
\end{equation*}%
for any such transformation $A$. The Pfaffian $\chi \left( A\right) $ of $A$
is a particular example; $\chi \left( A\right) =\widehat{\chi }\left(
\lambda _{1},...,\lambda _{s}\right) =\pm \lambda _{1}...\lambda _{s}$,
where the positive sign is chosen exactly when $e_{1},...,e_{2m}$ is a
positively oriented basis of $\mathbb{R}^{2s}$.

As we have seen in the proof of Lemma \ref{isometricDiskLemma} and Corollary %
\ref{isometricModificationCorollary}, the given nondegenerate transverse
Killing field $X$ with singular set $\Sigma $, its linearization restricts
to each $N_{x}\Sigma $ to be a Killing field. The restriction of its Lie
derivative to $N_{x}\Sigma $ is a nonsingular skew-symmetric automorphism $%
P_{x}\left( \left. \mathcal{L}_{X}\right\vert _{\Gamma \left( N\Sigma
\right) _{x}}\right) $, where $P_{x}:T_{x}M\rightarrow N_{x}\Sigma $ is the
orthogonal projection. Further we multiply the endomorphism by a positive
scalar $c_{x}$ so that the resulting skeigen-values $\left\{ \alpha
_{j}\right\} $ satisfy $\sum \alpha _{j}^{2}=1$ and each $\alpha _{j}$ is
nonzero. Let $\Lambda _{X}^{\nu }=c_{x}P_{x}\left( \left. \mathcal{L}%
_{X}\right\vert _{\Gamma \left( N\Sigma \right) _{x}}\right) $. We extend $%
\Lambda _{X}^{\nu }$ by zero on $T\Sigma $ to define the endomorphism $%
\Lambda _{X}:\left. TM\right\vert _{\Sigma }\rightarrow \left. TM\right\vert
_{\Sigma }$. By the results of Section \ref{StrTubNbhdSection}, and in
particular Corollary \ref{RIgidityIsomFlowsCorollary}, the skeigen-values of 
$\Lambda _{X}$ and of $\Lambda _{X}^{\nu }$ (i.e. the nonzero skeigen-values
of $\Lambda _{X}$) are constant on each connected component of $\Sigma $.
Let $\mu _{0}=0,$ $\mu _{1},...,\mu _{\tau }$ be the distinct skeigen-values
of $\Lambda _{X}$. Furthermore, $\left. TM\right\vert _{\Sigma }$ is the
direct sum of skeigen-bundles $\left. TM\right\vert _{\Sigma }=E_{0}\oplus
E_{1}\oplus ...\oplus E_{\tau }$, where $E_{0}=T\Sigma $ and%
\begin{equation*}
\left( E_{j}\right) _{x}=-\mu _{j}^{2}\text{ eigenspace of }\left( \Lambda
_{X}\right) _{x}^{2}
\end{equation*}%
For each $j\geq 1$, $E_{\lambda _{j}}$ can be endowed with the complex
structure $\left. \frac{1}{\mu _{j}^{2}\text{ }}\left( \Lambda _{X}\right)
^{2}\right\vert _{E_{\mu _{j}}}$ with induced orientation. We orient $%
E_{0}=T\Sigma $ so that the orientation agrees with the induced orientation
from $TM$. We set the real fiber dimension of $E_{j}$ to be $2m_{j}$, so
that $\sum_{j=0}^{\tau }m_{j}=m=\frac{1}{2}\dim M$. We now introduce forms $%
a_{j}$; in the case where $E_{0}$,..., $E_{\tau }$ 
are direct sums of line bundles, they are the first Chern forms (or, classes if considered
as elements of $H^{\ast }\left( \Sigma \right) $) of the line
bundle components. In general, let $%
a_{1},...,a_{m}$ be such that

\begin{enumerate}
\item The $i^{\text{th}}$ Pontryagin class of $E_{0}$ is the $i^{\text{th}}$
symmetric function of $a_{1}^{2},...,a_{m_{0}}^{2}$, and its Euler class is $%
a_{1}...a_{m_{0}}$.

\item For $i=1,...,\tau $, the $k^{\text{th}}$ Chern class of $E_{i}$ the $%
k^{\text{th}}$ elementary symmetric function of those $a_{j}^{2}$ such that $%
m_{0}+...+m_{i-1}+1\leq j\leq m_{0}+...+m_{i}$.
\end{enumerate}

Let $\lambda _{1},...,\lambda _{m}$ be the list of real numbers $\underset{%
m_{0}\text{ times}}{\underbrace{0,...,0}},\underset{m_{1}\text{ times}}{%
\underbrace{\mu _{1},...,\mu _{1}}},$ $...,\underset{m_{\tau }\text{ times}}{%
\underbrace{\mu _{\tau },...,\mu _{\tau }}}$, so that they are the
skeigen-values of $\Lambda _{X}$. We define%
\begin{equation*}
\psi \left( \Lambda _{X}\right) :=\widehat{\psi }\left( \lambda
_{1}+a_{1},...,\lambda _{m}+a_{m}\right) .
\end{equation*}%
One specific example we will use is%
\begin{equation*}
\chi \left( \Lambda _{X}^{\nu }\right) =\left( \lambda
_{m_{0}+1}+a_{m_{0}+1}\right) ...\left( \lambda _{m}+a_{m}\right) .
\end{equation*}

There is a technical change we need to make in the case $\Sigma $ is a
point, in which case $TM=E_{1}\oplus ...\oplus E_{\tau }$, and it may be the
case that the orientation induced from the complex structures on the $E_{j}$
does not produce the given orientation of $TM$. In this case, we instead let%
\begin{equation*}
\psi \left( \Lambda _{X}\right) :=\widehat{\psi }\left( -\lambda
_{1}-a_{1},\lambda _{2}+a_{2},...,\lambda _{m}+a_{m}\right) .
\end{equation*}

\begin{theorem}
\label{mainTheorem}Let $\left( M,g\right) $ be a compact, oriented
Riemannian manifold of dimension $2m$ that is endowed with an oriented
singular Riemannian foliation $\mathcal{F}$. Let $X$ be a nondegenerate
transverse Killing vector field on $M$ whose span is $T\mathcal{F}$. Let $%
\phi $ be an $\mathrm{ad}\left( \mathrm{SO}\left( 2m\right) \right) $%
-invariant symmetric form of degree $m$ on $\mathfrak{o}\left( 2m\right) $.
Then the characteristic number ${\phi }\left( M\right) $ defined by $\phi $
satisfies%
\begin{equation*}
{\phi \left( M\right) }=\sum_{j}\frac{{\phi \left( \Lambda _{X}\right) }}{{%
\chi \left( \Lambda _{X}^{\nu }\right) }}\left[ \Sigma _{j}\right] ,
\end{equation*}%
where $\Sigma _{j}$ are the connected components of the singular stratum $%
\Sigma $ of $\mathcal{F}$.

\begin{proof}
The vector field $X$ globally generates a singular Riemannian flow. For $p$
near $\Sigma $, we replace $X$ with $\widetilde{X}=\frac{d}{dt}\left. \exp
^{\bot }\left( \exp \left( t\Lambda _{X}^{\nu }\right) \left( \exp ^{\bot
}\right) ^{-1}\left( p\right) \right) \right\vert _{t=0}$ where $\exp ^{\bot
}:N\Sigma \rightarrow M$ is the normal exponential map and $\exp :\mathfrak{%
so}\left( 2k\right) \rightarrow SO\left( 2k\right) $ is the Lie group
exponential with $2k=2m-\dim \Sigma $. The flow of this vector field is the
same as the flow of $X$, and in the metric $g_{1}$ from Theorem \ref%
{TubularRiemFlowIsIsometric}, $\widetilde{X}$ is an isometric flow near $%
\Sigma $. Lemma \ref{LocalizationLemma} shows that we need only calculate
each%
\begin{equation*}
-\lim_{\delta \rightarrow 0}\int_{\partial T_{\delta }\Sigma _{j}}\eta _{1}~.
\end{equation*}%
\newline
We refer to \cite[proof of Theorem C]{BaumCh} for the calculation of the
residue, where the calculation is local and only uses the fact $\widetilde{X}
$ is Killing in the small tubular neighborhood. Since the final formula of
the limit is the same for both $X$ and $\widetilde{X}$, the result follows.
\end{proof}
\end{theorem}

\begin{remark}
For the special case where $\Sigma _{j}$ is an isolated fixed point $p$, 
\begin{eqnarray*}
\frac{{\phi \left( \Lambda _{X}\right) }}{{\chi \left( \Lambda _{X}^{\nu
}\right) }}\left[ \Sigma _{j}\right]  &=&\frac{{\phi \left( \Lambda
_{X}\right) }}{{\chi \left( \Lambda _{X}^{\nu }\right) }}\left( p\right)  \\
&=&\frac{\widehat{\psi }\left( \lambda _{1}+a_{1},\lambda
_{2}+a_{2},...,\lambda _{m}+a_{m}\right) }{\left( \lambda _{1}+a_{1}\right)
\left( \lambda _{2}+a_{2}\right) ...\left( \lambda _{m}+a_{m}\right) }\left(
p\right) .
\end{eqnarray*}
\end{remark}

\begin{remark}
The theorem above can easily be adapted to the case where the characteristic
numbers come from the curvature of a more general foliated vector bundle
over $M$. In this case, $X$ acts canonically on such a bundle.
\end{remark}

\begin{example}
\label{RP_example}The following singular foliation is from \cite[Section 3.4]%
{Royo-P}. Consider the foliation on $S^{4}$ defined as follows. Let $%
v=\left( 
\begin{array}{c}
\alpha \\ 
\beta%
\end{array}%
\right) $ be an eigenvector of a symmetric matrix $B\in SL\left( 2,\mathbb{Z}%
\right) $ with positive irrational eigenvalues. We consider $S^{4}$ to be a
suspension of $S^{3}\subseteq \mathbb{C}^{2}$, and we foliate each $S^{3}$
by the curves $t\mapsto \left( \exp \left( it\alpha \right) z_{1},\exp
\left( it\beta \right) z_{2}\right) $. This nonsingular isometric flow on $%
S^{3}$ extends to an isometric flow of $S^{4}$, with two fixed points at the
poles. Note that each generic leaf closure of the flow is a two-dimensional
torus. A tubular neighborhood of such a torus is isometric to a solid torus
of the form $D^{2}\times T^{2}$, where $D^{2}$ is a two-dimensional disk,
and where the boundary of this tube is a (rectangular) $3$-torus $%
S^{1}\times T^{2}$. Choose two tubes $\mathrm{Tube}_{1}$ and $\mathrm{Tube}%
_{2}$ like this inside $S^{4}$ that are isometric and disjoint. We glue the
two boundary components of $S^{4}\setminus \left\{ \mathrm{Tube}_{1}\cup 
\mathrm{Tube}_{2}\right\} $ via the $3\times 3$ matrix $\left( 
\begin{array}{cc}
1 & 0 \\ 
0 & B%
\end{array}%
\right) $, which is a foliated diffeomorphism between the boundary
components. This is equivalent to attaching a handle. The result is the
manifold%
\begin{equation*}
M=\left\{ S^{4}\setminus \left\{ \mathrm{Tube}_{1}\cup \mathrm{Tube}%
_{2}\right\} \right\} /\sim ~,
\end{equation*}%
where the equivalence relation $\sim $ is given by the gluing map described
above. For a small interval $I$, we use the product metric on $\left(
\partial \mathrm{Tube}_{1}\right) \times I\cong \left( \partial \mathrm{Tube}%
_{2}\right) \times I$, and using a basic partition of unity (a partition of
unity that is constant on the leaves) we patch this to the original metric
on $S^{4}\setminus \left\{ \mathrm{Tube}_{1}\cup \mathrm{Tube}_{2}\right\} $
using Lemma \ref{patchingLemma}. The original foliation induces a singular
Riemannian flow on $M$ with this metric. It was shown in \cite[Section 3.4]%
{Royo-P} that this flow is not isometric. In fact, we certainly could attach
more handles as desired. In any case, we could now compute for example the
Euler characteristic of this manifold using Theorem~\ref{mainTheorem}. The
residue at each pole is 1, so that 
\begin{equation*}
\chi \left( M\right) =1+1=2.
\end{equation*}%
The same result could be obtained from the Hopf index theorem.
\end{example}

\begin{example}
\label{projective_example}Consider the manifold $\mathbb{C}P^{m}$, with
homogeneous coordinates $\left[ z_{0},...,z_{m}\right] $. Consider the
isometric flow parametrized by the curves $t\mapsto \left[ z_{0},\exp \left(
it\alpha _{1}\right) z_{1},...,\exp \left( it\alpha _{m}\right) z_{m}\right] 
$, where $\left( \alpha _{1},...,\alpha _{m}\right) $ is an eigenvector of a
specific matrix $A\in SL\left( m,\mathbb{Z}\right) $, where $\left\{ \alpha
_{1},...,\alpha _{m}\right\} $ is linearly independent over $Q$. This
isometric flow has $m+1$ fixed points $\left[ 1,0,...,0\right] $, $\left[
0,1,0,...,0\right] $, ...,$\left[ 0,0,...,0,1\right] $. Similar to the last
example, we note that generic leaf closures of the flow are $m$-dimensional
tori. A tubular neighborhood of such a torus is isometric to a tube of the
form $D^{m}\times T^{m}$, where $D^{m}$ is a $m$-dimensional disk, and where
the boundary of this tube is of the form $S^{m-1}\times T^{m}$. Choose two
tubes $\mathrm{Tube}_{1}$ and $\mathrm{Tube}_{2}$ like this inside $\mathbb{C%
}P^{m}$ that are isometric and disjoint. We glue the two boundary components
of $\mathbb{C}P^{m}\setminus \left\{ \mathrm{Tube}_{1}\cup \mathrm{Tube}%
_{2}\right\} $ via the map $\mathrm{id}\times $ $A$, which is a foliated
diffeomorphism between the boundary components. The result is the manifold%
\begin{equation*}
M=\left\{ \mathbb{C}P^{m}\setminus \left\{ \mathrm{Tube}_{1}\cup \mathrm{Tube%
}_{2}\right\} \right\} /\sim ~,
\end{equation*}%
where the equivalence relation $\sim $ is given by the gluing map described
above. For a small interval $I$, we use the product metric on $\left(
\partial \mathrm{Tube}_{1}\right) \times I\cong \left( \partial \mathrm{Tube}%
_{2}\right) \times I$, and using a basic partition of unity we patch this to
the original metric on $\mathbb{C}P^{m}\setminus \left\{ \mathrm{Tube}%
_{1}\cup \mathrm{Tube}_{2}\right\} $ using Lemma \ref{patchingLemma}. The
original foliation induces a singular Riemannian flow on $M$ with this
metric. Similar to what is shown in \cite[Section 3.4]{Royo-P}, we can see
that this flow is not isometric. However, we may apply Theorem~\ref%
{mainTheorem} to compute the signature of $M$. On a small neighborhood of
each singular point $[0,\ldots ,z_{j}=1,0,\ldots ,0]$, the foliation has the
form of the flow 
\begin{equation*}
(z_{0},z_{1},\ldots ,\widehat{z_{j}},\ldots ,z_{m})\mapsto (\exp (-it\alpha
_{j})z_{0},\exp (it(\alpha _{1}-\alpha _{j}))z_{1},\ldots ,\widehat{z_{j}}%
,\ldots ,\exp (it(\alpha _{m}-\alpha _{j}))z_{m}),
\end{equation*}%
letting $\alpha _{0}=0$. Then the residue calculation for the signature
gives (see Corollary~\ref{SignatureCor} below), 
\begin{equation*}
\sigma (M)=\sum_{j=0}^{m}\prod_{i=0,\neq j}^{m}\mathrm{sgn}\left( \alpha
_{i}-\alpha _{j}\right) =\left\{ 
\begin{array}{ll}
1\quad & \text{if $m$ is even,} \\ 
0\quad & \text{if $m$ is odd}%
\end{array}%
\right. .
\end{equation*}%
We see that the surgery did not alter the signature.
\end{example}

\section{The non-orientable case.}

\vspace{0in}Even if the closed manifold $M$ is orientable, it is possible
that there exists a singular Riemannian foliation on it that is not
orientable; see Example \ref{nonorientable_example}. In this case, it is
easy to modify the argument of the paper to compute characteristic numbers
of the manifold using a calculation at the singular stratum of the foliation.

\begin{lemma}
\label{doubleCoverLemma}Let $\left( M,\mathcal{F},g\right) $ be a
non-orientable, one-dimensional singular Riemannian foliation of an
oriented, Riemannian manifold $M$. Then there exists an oriented singular
Riemannian flow $\left( \widetilde{M},\widetilde{\mathcal{F}},\pi ^{\ast
}g\right) $ on a double cover $\pi :\widetilde{M}\rightarrow M$ such that
the regular leaves of $\widetilde{\mathcal{F}}$ are double covers of the
regular leaves of $\mathcal{F}$.
\end{lemma}

\begin{proof}
Every foliation chart $U\subseteq M$ contains a dense subset saturated by
one-dimensional plaques of $\mathcal{F}$. Thus, $U$ can be endowed with one
of two possible leafwise orientations. We construct the foliation $\left( 
\widetilde{M},\widetilde{\mathcal{F}}\right) $ by making the foliation
charts from sets of the form $U\times o\left( U\right) $, where $U$ is a
foliation chart and $o\left( U\right) $ is a choice of leafwise orientation
of $U$. The plaques are the subsets of the form $P\times o\left( U\right) $,
where $P$ is a plaque of $\mathcal{F}$ in $U$. The properties follow easily
from the definition, and the orientations $o\left( U\right) $ produce the
leafwise orientation of $\widehat{\mathcal{F}}$.
\end{proof}

The fact that characteristic numbers are multiplicative over finite covers
yields the following.

\begin{theorem}
\label{nonorientableCase}\vspace{0in}Let $\left( M,\mathcal{F},g\right) $ be
a non-orientable, one-dimensional singular Riemannian foliation of a
compact, oriented, Riemannian manifold $M$ of dimension $2m$. Let $\left( 
\widetilde{M},\widetilde{\mathcal{F}},\pi ^{\ast }g\right) $ be the
corresponding double cover as in Lemma \ref{doubleCoverLemma}. Let $X$ be a
nondegenerate transverse Killing vector field on $\widetilde{M}$ whose span
is $T\widetilde{\mathcal{F}}$. Let $\phi $ be an $\mathrm{ad}\left( \mathrm{%
SO}\left( 2m\right) \right) $-invariant symmetric form of degree $m$ on $%
\mathfrak{o}\left( 2m\right) $. Then the characteristic number ${\phi }%
\left( M\right) $ defined by $\phi $ satisfies%
\begin{equation*}
{\phi \left( M\right) }=\frac{1}{2}\sum_{j}\frac{{\phi \left( \Lambda
_{X}\right) }}{{\chi \left( \Lambda _{X}^{\nu }\right) }}\left[ \widetilde{%
\Sigma _{j}}\right] ,
\end{equation*}%
where $\widetilde{\Sigma _{j}}$ are the connected components of the singular
stratum $\widetilde{\Sigma }$ of $\widetilde{\mathcal{F}}$.
\end{theorem}

\begin{corollary}
\label{EulerCor}Let $\left( M,\mathcal{F},g\right) $ be a possibly
nonorientable one-dimensional singular Riemannian foliation of a compact,
oriented, Riemannian manifold $M$ of dimension $2m$. Then the Euler
characteristic satisfies $\chi \left( M\right) =\sum_{j}\chi \left( \Sigma
_{j}\right) $, where the sum is over the components of the singular stratum
of the foliation.
\end{corollary}

\begin{proof}
From Theorem \ref{TubularRiemFlowIsIsometric}, the local calculation is the
same as it is for isometric flows, for which this formula is already known 
\cite{Kob1}.
\end{proof}

\begin{remark}
This formula was already known (\cite[Theorem C]{GalRad}) in the case where
the leaves are closed.
\end{remark}

Let $\left( M,\mathcal{F},g\right) $ be a possibly nonorientable
one-dimensional singular Riemannian foliation of a compact, oriented,
Riemannian manifold $M$ of dimension $4\ell $ whose singular stratum
consists of isolated points $p_{1},...,p_{\alpha }$. For each $j$, by
Theorem \ref{TubularRiemFlowIsIsometric}, there exists a Killing vector
field $X_{j}$ such that the restriction of $\mathcal{F}$ to a ball centered
at $p_{j}$ is the flow of $X_{j}$. We define the \textbf{index} $\varepsilon
_{j}$ of the singular point $p_{j}$ as follows. The isometric flow
corresponding to $X_{j} $ has the local form (\ref{formOfIsometricFlow}) on $%
\mathbb{C}^{2\ell }$ with constants $\alpha _{1},...,\alpha _{2\ell }$. We
define the sign 
\begin{equation*}
\varepsilon _{j}=\pm \prod_{i=1}^{2\ell }\mathrm{sgn}\left( \alpha
_{i}\right) ,
\end{equation*}%
where the $\pm $ is chosen depending on whether the orientation of $\mathbb{C%
}^{2\ell }$ agrees with the orientation of $M$ at $p_{j}$ or not. It is easy
to check that this number is independent of the choice of coordinates and of
the choice of $X_{j}$. Even if the orientation of $X_{j}$ is reversed, the
product $\prod_{i=1}^{2\ell }\mathrm{sgn}\left( \alpha _{i}\right) $ is
invariant.

\begin{corollary}
\label{SignatureCor}Let $\left( M,\mathcal{F},g\right) $ be a possibly
nonorientable one-dimensional singular Riemannian foliation of a compact,
oriented, Riemannian manifold $M$ of dimension $4\ell $ whose singular
stratum consists of isolated points. Then the signature satisfies $\sigma
\left( M\right) =\sum_{j}\varepsilon _{j}$, where the sum is over the
singular points and $\varepsilon _{j}$ is the index of the $j^{\text{th}}$
singular point.
\end{corollary}

This result is new for the case where $\left( M,\mathcal{F},g\right) $ is
not an isometric flow.

Below is an example where an orientable 4-manifold is endowed with a
one-dimensional, non-orientable singular Riemannian foliation.

\begin{example}
\label{nonorientable_example}Consider a nonorientable disk bundle over a
Klein bottle. More specifically, let $K=\mathbb{R}^{2}\diagup \left\langle
\tau ,\sigma \right\rangle $, where $\tau \left( x,y\right) =\left(
x,y+1\right) $, $\sigma \left( x,y\right) =\left( x+1,1-y\right) $. Consider 
$D\times \mathbb{R}^{2}$, where $D$ is the unit disk in $\mathbb{C}$. We
extend $K$ to be a disk bundle $B$ over $K$ by letting $B=D\times \mathbb{R}%
^{2}\diagup \left\langle \tau ,\sigma \right\rangle $, where $\tau \left(
z,x,y\right) =\left( z,x,y+1\right) $ and $\sigma \left( z,x,y\right)
=\left( \overline{z},x+1,1-y\right) $. Then $B$ is an orientable $4$%
-manifold with boundary $\partial B=\partial D\times \mathbb{R}^{2}\diagup
\left\langle \tau ,\sigma \right\rangle $. We foliate $B$ by circles of the
form $C_{r,x,y}=\left\{ \left( z,x,y\right) \diagup \left\langle \tau
,\sigma \right\rangle :\left\vert z\right\vert =r\right\} $ with $0\leq
r\leq 1$. Now we glue two copies of $B$ with opposite orientation together
to form and orientable $4$-manifold $M=B\sqcup _{\partial B}B^{\prime }$,
which can be endowed with a metric such that the circles $C_{r,x,y}$ form a
singular Riemannian foliation that is not orientable. Since the double cover 
$\widetilde{M}$ is diffeomorphic to $T^{2}\times S^{2}$, all characteristic
numbers of $M$ are zero.
\end{example}


\begin{thebibliography}{99}
\bibitem{AlexPlus} M. M. Alexandrino, R. Briquet, and D. T\"{o}ben, \emph{%
Progress in the theory of singular Riemannian foliations}, Differential
Geom. Appl. \textbf{31} (2013), no. 2, 248--267.


\bibitem{APS3} M. F. Atiyah and I. M. Singer, \emph{The index of elliptic
operators III}, Ann. of Math. (2) \textbf{87} (1968) 546--604.

\bibitem{BaumCh} P. Baum and J. Cheeger, \emph{Infinitesimal isometries and
Pontryagin numbers}, Topology \textbf{8} (1969) 173--193.

\bibitem{bott1} R. Bott, \emph{Vector fields and characteristic numbers},
Michigan Math. J. \textbf{14} (1967), 231--244.

\bibitem{bott2} R. Bott, \emph{A residue formula for holomorphic
vector-fields}, J. Differential Geometry \textbf{1} (1967) 311--330.



\bibitem{Carr1.5} Y. Carri\`{e}re, \emph{Les propri\'{e}t\'{e}s topologiques
des flots riemanniens retrouv\'{e}es \`{a} l'aide du th\'{e}or\`{e}me des
vari\'{e}t\'{e}s presque plates}, Math. Z. \textbf{186} (1984), no. 3,
393--400.


%
%

\bibitem{GalRad} F. Galaz-Garcia and M. Radeschi, \emph{Singular Riemannian
foliations and applications to positive and non-negative curvature}, J.
Topol. \textbf{8} (2015), no. 3, 603--620.

\bibitem{GrGr} D. Gromoll and K. Grove, \emph{One-dimensional metric
foliations in constant curvature spaces}, Differential Geometry and Complex
Analysis, H. E. Rauch memorial volume, Springer, Berlin, 1985, 165--167.

%

\bibitem{Kob1} S. Kobayashi, \emph{Fixed points of isometries}, Nagoya Math.
J. \textbf{13} (1958), 63--68.

%

\bibitem{Mei} X.-M. Mei, \emph{Note on the residues of the singularities of
a Riemannian foliation}, Proc. Amer. Math. Soc. \textbf{89} (1983), no. 2,
359--366.

\bibitem{MenRad} R. Mendes and M. Radeshi, \emph{Smooth basic functions},
preprint, arXiv:1511.06174.

\bibitem{Mo} P. Molino, \emph{Riemannian foliations}, Progress in
Mathematics \textbf{73}, Boston: Birkh\"{a}user Boston, Inc., 1988.


\bibitem{MolPier} P. Molino, and M. Pierrot, \emph{Th\'{e}or\`{e}mes de
slice et holonomie des feuilletages riemanniens singuliers}, Ann. Inst.
Fourier (Grenoble) \textbf{37} (1987), no. 4, 207--223.

\bibitem{Re1} B. L. Reinhart, \emph{Differential geometry of foliations: The
fundamental integrability problem}, Ergebnisse der Mathematik und ihrer
Grenzgebiete, vol. 99, Springer-Verlag, Berlin, 1983.

\bibitem{Royo-P} J. I. Royo Prieto, \emph{Estudio Cohomol\'{o}gico de flujos
riemannianos}, Ph.D. Thesis, University of the Basque Country UPV/EHU, 2003.



\bibitem{Stef} P. Stefan, \emph{Accessible sets, orbits, and foliations with
singularities}, Proc. London Math. Soc. (3) \textbf{29} (1974), 699--713.

\bibitem{Suss} H. J. Sussmann, \emph{Orbits of families of vector fields and
integrability of distributions}, Trans. Amer. Math. Soc. \textbf{180}
(1973), 171--188.

%
\end{thebibliography}
\end{document}